\documentclass[12pt]{article}

\usepackage{epsfig}
\usepackage{latexsym}
\usepackage{amsmath}
\usepackage{amsfonts}
\usepackage{amssymb}
\usepackage{graphicx}%
\usepackage{algorithmic}
\usepackage{color}

\usepackage{hyperref} 

\hypersetup{
    colorlinks=true,            
    linkcolor=black,              
    citecolor=black,            
    filecolor=black,   
    urlcolor=black            
}

\makeatletter

\newcommand{\Rmnum}[1]{\expandafter\@slowromancap\romannumeral #1@}
\makeatother


\begin{document}
	
	\pagestyle{myheadings} \markright{\sc Covering a supermodular-like function in a mixed hypergraph\hfill} \thispagestyle{empty}
	
	\newtheorem{theorem}{Theorem}[section]
	\newtheorem{corollary}[theorem]{Corollary}
	\newtheorem{definition}[theorem]{Definition}
	\newtheorem{guess}[theorem]{Conjecture}
	\newtheorem{claim}[theorem]{Claim}
	\newtheorem{problem}[theorem]{Problem}
	\newtheorem{question}[theorem]{Question}
	\newtheorem{lemma}[theorem]{Lemma}
	\newtheorem{proposition}[theorem]{Proposition}
	\newtheorem{fact}[theorem]{Fact}
	\newtheorem{acknowledgement}[theorem]{Acknowledgement}
	\newtheorem{algorithm}[theorem]{Algorithm}
	\newtheorem{axiom}[theorem]{Axiom}
	\newtheorem{case}[theorem]{Case}
	\newtheorem{conclusion}[theorem]{Conclusion}
	\newtheorem{condition}[theorem]{Condition}
	\newtheorem{conjecture}[theorem]{Conjecture}
	\newtheorem{criterion}[theorem]{Criterion}
	\newtheorem{example}[theorem]{Example}
	\newtheorem{exercise}[theorem]{Exercise}
	\newtheorem{notation}[theorem]{Notation}
	\newtheorem{observation}[theorem]{Observation}
	\newtheorem{solution}[theorem]{Solution}
	\newtheorem{summary}[theorem]{Summary}
	
	\newtheorem{thm}[theorem]{Theorem}
	\newtheorem{prop}[theorem]{Proposition}
	\newtheorem{defn}[theorem]{Definition}

	\newtheorem{lem}[theorem]{Lemma}
	\newtheorem{con}[theorem]{Conjecture}
	\newtheorem{cor}[theorem]{Corollary}

	\newenvironment{proof}{\noindent {\bf
			Proof.}}{\rule{3mm}{3mm}\par\medskip}
	\newcommand{\remark}{\medskip\par\noindent {\bf Remark.~~}}
	\newcommand{\pp}{{\it p.}}
	\newcommand{\de}{\em}

	\newcommand{\g}{\mathrm{g}}

	\newcommand{\qf}{Q({\cal F},s)}
	\newcommand{\qff}{Q({\cal F}',s)}
	\newcommand{\qfff}{Q({\cal F}'',s)}
	\newcommand{\f}{{\cal F}}
	\newcommand{\ff}{{\cal F}'}
	\newcommand{\fff}{{\cal F}''}
	\newcommand{\fs}{{\cal F},s}
	\newcommand{\cs}{\chi'_s(G)}
	
	\newcommand{\G}{\Gamma}
	\newcommand{\wrt}{with respect to }
	\newcommand{\mad}{{\rm mad}}
	\newcommand{\col}{{\rm col}}
	\newcommand{\gcol}{{\rm gcol}}
	
	\newcommand*{\ch}{{\rm ch}}
	\newcommand*{\ra}{{\rm ran}}
	\newcommand{\co}{{\rm col}}
	\newcommand{\sco}{{\rm scol}}
	\newcommand{\wc}{{\rm wcol}}
	\newcommand{\dc}{{\rm dcol}}
	\newcommand*{\ar}{{\rm arb}}
	\newcommand*{\ma}{{\rm mad}}
	\newcommand{\di}{{\rm dist}}
	\newcommand{\tw}{{\rm tw}}
	\newcommand{\scol}{{\rm scol}}
	\newcommand{\wcol}{{\rm wcol}}
	\newcommand{\td}{{\rm td}}
	\newcommand{\edp}[2]{#1^{[\natural #2]}}
	\newcommand{\epp}[2]{#1^{\natural #2}}
	\newcommand*{\ind}{{\rm ind}}
	\newcommand{\red}[1]{\textcolor{red}{#1}}
	
	\def\C#1{|#1|}
	\def\E#1{|E(#1)|}
	\def\V#1{|V(#1)|}
	\def\iarb{\Upsilon}
	\def\ipac{\nu}
	\def\nul{\varnothing}

	\newcommand*{\QEDA}{\ensuremath{\blacksquare}}
	\newcommand*{\QEDB}{\hfill\ensuremath{\square}}

	
	
	
	
	\title{\Large\bf  Covering a supermodular-like function in a mixed hypergraph}
	
	\author{Hui Gao \\
		School of Mathematics\\
		China University of Mining and Technology\\
		Xuzhou, Jiangsu 221116, China\\
		E-mail: \texttt{gaoh1118@yeah.net}}

	\maketitle

\begin{abstract}
	In this paper, we solve a conjecture by Szigeti in [Matroid-rooted packing of arborescences, submitted], which characterizes a mixed hypergraph $\mathcal{F}=(V, \mathcal{E} \cup \mathcal{A})$ having an orientation $\overrightarrow{\mathcal{E}}$ of $\mathcal{E}$ such that $e_{\overrightarrow{\mathcal{E}} \cup \mathcal{A}} (\mathcal{P}) \geq \sum_{X \in \mathcal{P}}h(X) -b(\cup \mathcal{P})$ for every subpartition $\mathcal{P}$ of $V$, where $h$ is an integer-valued, intersecting supermodular function on $V$ and $b$ a submodular function on $V$. As a corollary, another conjecture in the same paper is confirmed, which characterizes a mixed hypergraph having a packing of mixed hyperarborescences such that their roots form a basis in a given matroid, each vertex $v$ belongs to exactly $k$ of them and is the root of at least $f(v)$ and at most $g(v)$ of them.
\end{abstract}



\section{Introduction}
In this paper, all graphs, digraphs, hypergraphs or dypergraphs may have multiple edges, arcs, hyperedges or dyperedges respectively, but no loops.

A digraph $F$ is called an arborescence rooted at a vertex $r \in V(F)$ (or $r$-arborescence) if its underlying graph is a tree and for each $v \in V(F)$, there exists a directed path from $r$ to $v$. The vertex $r$ is called the root of $F$. Let $D=(V, A)$ be a digraph. A set of arc-disjoint subgraphs of $D$ is called a packing of subgraphs.
Packing arborescences in digraphs is a fundamental and well-studied problem in graph theory. The basic problem of packing of spanning arborescences with fixed roots is due to Edmonds~\cite{E-73}.

A multiset of vertices in $V$ may contain multiple occurrences of elements. For a multiset $S$ of vertices in $V$ and a subset $X$ of $V$ , $S_{X}$ denotes the multiset consisting of the elements of $X$ with the same multiplicities as in $S$.

\begin{thm}[\cite{E-73}]\label{Edmonds}
Let $D = (V, A)$ be a digraph and $S$ a multiset of vertices in $V$. There exists a packing of spanning $s$-arborescences $(s \in S)$ in $D$ if and only if
\begin{equation}
d^{-}_{A}(X) \geq |S_{V-X}|~~~\text{for every $\emptyset \neq X \subseteq V$.}
\end{equation}
\end{thm}

Denote by $\mathbb{Z}_{+}$ the set of nonnegative integers and let $k \in Z_{+}$. Theorem~\ref{Edmonds} implies a characterization of the existence of a packing of $k$-regular $s$-arborescences ($s \in S$) in $D$, where $k$-regular means that each vertex in $D$ must belong to exactly $k$ of the arborescences.

\begin{thm}[\cite{E-73}]
Let $D = (V, A)$ be a digraph, $k \in \mathbb{Z}_{+}$, $S$ a multiset of vertices in $V$. There exists a $k$-regular packing of $s$-arborescences ($s \in S$) in $D$ if and only if
\begin{equation}
k \geq |S_{v}|~~~\text{for every $v \in V$,}
\end{equation}
\begin{equation}
d_{A}^{-}(X) \geq k-|S_{X}|~~~\text{for every $\emptyset \neq X \subseteq V$.}
\end{equation}
\end{thm}

Later Frank~\cite{F-78} solved the problem of packing of spanning arborescences with flexible roots. In fact, Frank~\cite{F-78} (and independently Cai~\cite{C-83}) provided a result on $(f, g)$-bounded packings of $k$ spanning arborescences where $f(v)$ is a lower bound and $g(v)$ is an upper bound on the number of $v$-arborescences in the packing for every vertex $v$.

A set of disjoint subsets of $V$ is called a subpartition of $V$. For a subpartition $\mathcal{P}$ of $V$, we denote by $\cup \mathcal{P}$ the vertex set which is the union of the members of $\mathcal{P}$. Denote by $\mathbb{Z}^{V}_{+}$ the set of nonnegative integer valued function on $V$.

\begin{thm}[\cite{C-83,F-78}]
Let $D = (V, A)$ be a digraph, $f, g \in \mathbb{Z}_{+}^{V}$, and $k \in \mathbb{Z}_{+}$. There exists an $(f, g)$-bounded packing of $k$ spanning arborescences in $D$ if and only if
\begin{equation}
g(v) \geq f(v)~~~\text{for every $v \in V$,}
\end{equation}
\begin{equation}
e_{A}(\mathcal{P}) \geq k|\mathcal{P}|-\min\{k-f(V \setminus \cup \mathcal{P}), g(\cup \mathcal{P}) \}~~~\text{for every subpartition $\mathcal{P}$ of $V$.}
\end{equation}
\end{thm}

Durand de Gevigney, Nguyen and Szigeti~\cite{DNS-13} considered the problem of matroid-based packing of arborescences. In this problem here are given a digraph $D = (V, A)$, a multiset $S$ of vertices in $V$ and a matroid $M$ on $S$, and what is wanted is a packing $\mathcal{B}$ of (not necessarily spanning) arborescences such that for every $v \in V$ , the set of roots of the arborescences in $\mathcal{B}$ that contain $v$ must form a basis of $M$. They gave in~\cite{DNS-13} a characterization of the existence of a matroid-based packing of arborescences.

\begin{thm}[\cite{DNS-13}]
Let $D = (V, A)$ be a digraph, $S$ a multiset of vertices in $V$ , and $M = (S, r_{M})$ a matroid. There exists an $M$-based packing of arborescences in $D$ if and only if
\begin{equation}
d^{-}_{A}(X) \geq r_{M}(S)-r_{M}(S_{X})~~~\text{for every $\emptyset \neq X \subseteq V$.}
\end{equation}
\end{thm}

The above problems were generalized for mixed graphs by Frank~\cite{F-78}, Gao,Yang~\cite{GY21}, and Fortier et al.~\cite{FKLST-18}, also for directed hypergraphs by Frank, Kir\'{a}ly, Kir\'{a}ly~\cite{FKK-03}, H$\ddot{\rm{o}}$rsch, Szigeti~\cite{HS-21}, and~\cite{FKLST-18}, and even for mixed hypergraphs in~\cite{FKK-03,HS-21}, respectively

Szigeti~\cite{S-24} introduced a new problem on packing of arborescences with a new matroid constraint. Given a digraph $D = (V, A)$, a multiset S of vertices in $V$ and a matroid $M$ on $S$, a packing $\mathcal{B}$ of (not necessarily spanning) arborescences is called $M$-rooted if the set of roots of the arborescences in $\mathcal{B}$ is a basis of $M$. Note that if each arborescence in $\mathcal{B}$ is spanning then the condition of $M$-based packing coincides with the condition of $M$-rooted packing.

Let $\mathcal{D} = (V, \mathcal{A})$ be a dypergraph, where $V$ is the set of vertices and $\mathcal{A}$ is the set of dyperedges of $\mathcal{D}$. A dyperedge is an ordered pair $(Z, z)$ such that $z$ is a vertex in $V$, called the head, and $Z$ is a non-empty subset of $V -z$, called the set of tails. For $X \subseteq V$, we say that a dyperedge $(Z, z) \in \mathcal{A}$ enters $X$ if $z \in X$ and $Z \setminus X \neq \emptyset$.  The operation that replaces a dyperedge $(Z, z)$ by an arc $yz$ where $y \in Z$ is called trimming. We say that $\mathcal{D}$ is an $s$-hyperarborescence, if $\mathcal{D}$ can be trimmed to an $s$-arborescence. We say that $\mathcal{D}$ has an $M$-rooted/$(f, g)$-bounded/$k$-regular packing of hyperarborescences if $\mathcal{D}$ can be trimmed to a digraph that has an $M$-rooted/$(f, g)$-bounded/$k$-regular packing of arborescences.

Let $\mathcal{F} = (V, \mathcal{E} \cup \mathcal{A})$ be a mixed hypergraph, where $V$ is the set of vertices, $\mathcal{E}$ is the set of hyperedges and $\mathcal{A}$ is the set of dyperedges of $\mathcal{F}$. A hyperedge is a subset of $V$ of size at least two. A hyperedge $X$ enters a subset $Y$ of $V$ if $X \cap Y \neq \emptyset \neq X \setminus Y$. By orienting a hyperedge $X$, we mean the operation that replaces the hyperedge $X$ by a dyperedge $(X -x, x)$ for some $x \in X$. A mixed hypergraph that has an orientation that is an $s$-hyperarborescence is called a mixed $s$-hyperarborescence. By a packing of mixed subhypergraphs in $\mathcal{F}$ we mean a set of mixed subhypergraphs that are hyperedge- and dyperedge-disjoint. For a subpartition $\mathcal{P}$ of subsets of $V$ , we denote by $e_{\mathcal{E} \cup \mathcal{A}}(\mathcal{P})$ the number of hyperedges in $\mathcal{E}$ and dyperedges in $\mathcal{A}$ that enter some member of $\mathcal{P}$.

Actually Szigeti~\cite{S-24} solved the problem of $M$-rooted $(f, g)$-bounded $k$-regular packing of hyperarborescences in dypergraphs by using the theory of generalized polymatroids. Denote $g_{k}=\min \{g, k\}$.

\begin{thm}[\cite{S-24}]
Let $\mathcal{D} = (V, \mathcal{A})$ be a dypergraph, $k \in \mathbb{Z}_{+}$, $f, g \in \mathbb{Z}_{+}^{V}$, $S$ a multiset of vertices in $V$ , and $M = (S, r_{M})$ a matroid. There exists an $M$-rooted $(f, g)$-bounded $k$-regular packing of hyperarborescences in $\mathcal{D}$ if and only if for all $U, W \subseteq V$ and all subpartitions $\mathcal{P}$ of $W$,
\begin{equation}\label{inq-gf}
g_{k}(v) \geq f(v)~~~\text{for every $v \in V$,}
\end{equation}
\begin{equation}\label{inq-gkrm}
r_{M}(S_{U})+g_{k}(V-U) \geq r_{M}(S),
\end{equation}
\begin{equation}
e_{\mathcal{A}}(\mathcal{P})+r_{M}(S_{U})+g_{k}(W-U) \geq k|\mathcal{P}|+f(U-W).
\end{equation}
\end{thm}

Furthermore, Szigeti obtain a characterization of the existence of packing of matroid-rooted $(f, g)$-bounded $k$-regular mixed arborescences (Theorem~\ref{Mfgkmixgraph}) by an orientation theorem (Theorem~\ref{orient-graph}).

\begin{thm}[\cite{S-24}]\label{Mfgkmixgraph}
Let $F = (V, E \cup A)$ be a mixed graph, $k \in \mathbb{Z}_{+}$, $f, g \in \mathbb{Z}_{+}^{V}$, $S$ a multiset of vertices in $V$, and $M = (S, r_{M})$ a matroid. There exists an $M$-rooted $(f, g)$-bounded $k$-regular packing of mixed arborescences in $F$ if and only if (\ref{inq-gf}) and (\ref{inq-gkrm}) hold and for all $W, U \subseteq V$ and all subpartitions $\mathcal{P}$ of $W$,
\begin{equation}
e_{E \cup A}(\mathcal{P}) + r_{M}(S_{U} ) + g_{k}(W - U) \geq  k|\mathcal{P}| + f(U - W).
\end{equation}
\end{thm}

We say that two subsets $X ,Y$ of $V$ are intersecting if $X \cap Y \neq \emptyset$. A set function $h$ on $V$ is intersecting supermodular if for all intersecting $X, Y \subseteq V$,
\begin{equation}
h(X)+h(Y) \leq h(X \cup Y)+h(X \cap Y).
\end{equation}
A set function $b$ on $V$ is submodular if for all $X, Y \subseteq V$,
\begin{equation}
b(X)+b(Y) \geq b(X \cup Y)+b(X \cap Y).
\end{equation}

\begin{thm}[\cite{S-24}]\label{orient-graph}
Let $F = (V, E \cup A)$ be a mixed graph, $h$ an integer-valued intersecting supermodular function on $V$, and $b$ an integer-valued submodular function on $V$. There exists an
orientation $\overrightarrow{E}$ of $E$ such that
\begin{equation}
e_{\overrightarrow{E} \cup A}(\mathcal{P}) \geq \sum_{X \in \mathcal{P}}h(X) - b(\cup \mathcal{P})~~~\text{for every subpartition $\mathcal{P}$ of $V$}
\end{equation}
if and only if
\begin{equation}
e_{E \cup A}(\mathcal{P}) \geq \sum_{X \in \mathcal{P}}h(X) - b(\cup \mathcal{P})~~~\text{for every subpartition $\mathcal{P}$ of $V$}
\end{equation}
\end{thm}

There remains two open conjectures, the problem of matroid-rooted $(f, g)$-bounded $k$-regular packings of $k$-regular mixed hyperarborescences (Conjecture~\ref{Mfgkmixhyper}) and a possible extension of Theorem~\ref{orient-graph} to mixed hypergraphs (Conjecture~\ref{orient-hyper}), which would imply Conjecture~\ref{Mfgkmixhyper} by the same proof of Theorem~\ref{Mfgkmixgraph}~\cite{S-24}.

\begin{con}[\cite{S-24}]\label{Mfgkmixhyper}
Let $\mathcal{F} = (V, \mathcal{E} \cup \mathcal{A})$ be a mixed hypergraph, $k \in \mathbb{Z}_{+}$, $f, g \in \mathbb{Z}^{V}_{+}$, $S$ a multiset of vertices in $V$ , and $M = (S, r_{M})$ a matroid. There exists an $M$-rooted $(f, g)$-bounded k-regular packing of mixed hyperarborescences in $\mathcal{F}$ if and only if (\ref{inq-gf}) and (\ref{inq-gkrm}) hold and for all $U, W \subseteq V$ and subpartition $\mathcal{P}$ of $W$,
\begin{equation}
e_{\mathcal{E} \cup \mathcal{A}}(\mathcal{P}) + r_{M}(S_{U} ) + g_{k}(W - U) \geq k|\mathcal{P}| + f(U - W).
\end{equation}
\end{con}

\begin{con}[\cite{S-24}]\label{conj-orient}\label{orient-hyper}
Let $\mathcal{F}=(V , \mathcal{E} \cup \mathcal{A})$ be a mixed hypergraph, $h$ an integer-valued, intersecting supermodular function on $V$ and $b$ a submodular function on $V$. There exists an orientation $\overrightarrow{\mathcal{E}}$ of $\mathcal{E}$ such that
\begin{equation}\label{efficiency}
e_{\overrightarrow{\mathcal{E}} \cup \mathcal{A}}(\mathcal{P}) \geq \sum_{X \in \mathcal{P}}h(X)-b(\cup \mathcal{P})~~~ \text{for every subpartition $\mathcal{P}$ of $V$}
\end{equation}
if and only if
\begin{equation}\label{sufficiency}
e_{\mathcal{E} \cup \mathcal{A}}(\mathcal{P}) \geq \sum_{X \in \mathcal{P}}h(X)-b(\cup \mathcal{P}) ~~~\text{for every subpartition $\mathcal{P}$ of $V$.}
\end{equation}
\end{con}

In this paper,  we prove Conjecture~\ref{orient-hyper}.

\section{Proof of Conjecture~\ref{conj-orient}}

The necessity is obtained from the fact every orientation $\overrightarrow{\mathcal{E}}$ of $\mathcal{E}$ and every subpartition $\mathcal{P}$ of $V$, we have  $e_{\mathcal{E} \cup \mathcal{A}} (\mathcal{P}) \geq e_{\overrightarrow{\mathcal{E}} \cup \mathcal{A}}(\mathcal{P})$.

To prove the sufficiency, we orient each hyperedge in $\mathcal{E}$ one by one such that  (\ref{sufficiency}) still holds after the orientation. Finally, we obtain an orientation of $\mathcal{E}$ such that (\ref{efficiency}) holds.

Fix a hyperedge $Y \in \mathcal{E}$ and we will show how to orient $Y$. Let
$$\mathcal{S}:=\{ \mathcal{P} \in \Omega(V): \text{ $Y$ enters $\cup \mathcal{P}$ and}~e_{\mathcal{E} \cup \mathcal{A}}(\mathcal{P})= \sum_{X \in \mathcal{P}}h(X)-b(\cup \mathcal{P})\}.$$

If $\mathcal{S} =\emptyset$, then for any subpartition $\mathcal{P} \in \Omega(V)$, two cases happen:
\begin{itemize}
\item[(i)] $Y \subseteq \cup \mathcal{P}$ or $V \setminus \cup \mathcal{P}$;
\item[(ii)] $e_{\mathcal{E} \cup \mathcal{A}}(\mathcal{P})> \sum_{X \in \mathcal{P}}h(X)-b(\cup \mathcal{P})$.
\end{itemize}
Orient $Y$ in any direction to $\overrightarrow{Y}$ and we have $(\ref{sufficiency})$ still holds after the orientation since for Case (i), $e_{(\mathcal{E}-Y) \cup (\mathcal{A}+\overrightarrow{Y}) }= e_{\mathcal{E} \cup \mathcal{A}}$ and for Case (ii), $e_{(\mathcal{E}-Y) \cup (\mathcal{A}+\overrightarrow{Y}) } \geq  e_{\mathcal{E} \cup \mathcal{A}}-1$. In the following, suppose $\mathcal{S} \neq \emptyset$.

Let $\mathcal{P}_{1}, \mathcal{P}_{2} \in \mathcal{S}$ and $\mathcal{P} = \mathcal{P}_{1} \uplus \mathcal{P}_{2}$. Note that $\mathcal{P}$ covers each vertex in $(\cup \mathcal{P}_{1}) \cap (\cup \mathcal{P}_{2}) $ twice and each vertex in $ ((\cup \mathcal{P}_{1}) \cup (\cup \mathcal{P}_{2}) ) \setminus ( (\cup \mathcal{P}_{1}) \cap (\cup \mathcal{P}_{2}) ) $ once. Using the usual uncrossing technique for $\mathcal{P}$, we obtain a laminar family $\mathcal{P}'$ that covers  each vertex in $(\cup \mathcal{P}_{1}) \cap (\cup \mathcal{P}_{2}) $ twice and each vertex in $ ((\cup \mathcal{P}_{1}) \cup (\cup \mathcal{P}_{2}) ) \setminus ( (\cup \mathcal{P}_{1}) \cap (\cup \mathcal{P}_{2}) ) $ once. Then $\mathcal{P}'$ can be decomposed into a partition $\mathcal{P}_{3}$  of $(\cup \mathcal{P}_{1}) \cap (\cup \mathcal{P}_{2}) $ and a partition $\mathcal{P}_{4}$ of $(\cup \mathcal{P}_{1}) \cup (\cup \mathcal{P}_{2}) $. Since $h$ is intersecting supermodular and $b$ is submodular, we have
\[
\sum_{X \in \mathcal{P}_{1}} h(X) +\sum_{X \in \mathcal{P}_{2}} h(X) = \sum_{X \in \mathcal{P}} h(X) \leq  \sum_{X \in \mathcal{P}'} h(X)=\sum_{X \in \mathcal{P}_{3}} h(X) +\sum_{X \in \mathcal{P}_{4}} h(X) ~\text{and}
\]
\[
b(\cup \mathcal{P}_{1}) +b(\cup \mathcal{P}_{2}) \geq  b(\cup \mathcal{P}_{3}) +b(\cup \mathcal{P}_{4}).
\]
Thus
\begin{equation}\label{ineq-hb}
\sum_{X \in \mathcal{P}_{1}} h(X)-b(\cup \mathcal{P}_{1}) +\sum_{X \in \mathcal{P}_{2}} h(X)-b(\cup \mathcal{P}_{2}) \leq \sum_{X \in \mathcal{P}_{3}} h(X)-b(\cup \mathcal{P}_{3}) +\sum_{X \in \mathcal{P}_{4}} h(X)-b(\cup \mathcal{P}_{4}).
\end{equation}

\begin{claim}\label{claim-supseteq}
$(\mathcal{E}(\mathcal{P}_{1}) \cup \mathcal{A}(\mathcal{P}_{1}) ) \uplus (\mathcal{E}(\mathcal{P}_{2}) \cup \mathcal{A}(\mathcal{P}_{2}) ) \supseteq (\mathcal{E}(\mathcal{P}_{3}) \cup \mathcal{A}(\mathcal{P}_{3}) ) \uplus (\mathcal{E}(\mathcal{P}_{4}) \cup \mathcal{A}(\mathcal{P}_{4}) )$.
\end{claim}
\begin{proof}
Let $Y_{0} \in \mathcal{E}$. If $Y_{0} \notin \mathcal{E}(\mathcal{P}_{1})$, then we have (i) $Y_{0} \subseteq V \setminus \cup \mathcal{P}_{1}$ and thus $Y_{0} \notin  \mathcal{E}( \mathcal{P}_{3})$; or (ii)  $Y_{0} $ is contained in some member of $\mathcal{P}_{1}$ and thus $Y_{0} \notin \mathcal{E}(\mathcal{P}_{4})$. The same discussion goes for the case that $Y_{0} \notin \mathcal{E}(\mathcal{P}_{2})$. If $Y_{0} \notin \mathcal{E}(\mathcal{P}_{1}) \cup \mathcal{E}(\mathcal{P}_{2})$, then we have (i) $Y_{0} \subseteq V \setminus (\cup \mathcal{P}_{1}) \cup ( \cup \mathcal{P}_{2})$; or (ii) $Y_{0} \subseteq (\cup \mathcal{P}_{1}) \setminus (\cup \mathcal{P}_{2})$ and $Y_{0}$ is contained in some member of $\mathcal{P}_{1}$; or (iii) $Y_{0} \subseteq (\cup \mathcal{P}_{2}) \setminus (\cup \mathcal{P}_{1})$ and $Y_{0}$ is contained in some member of $\mathcal{P}_{2}$. All the three cases above induce that $Y_{0} \notin \mathcal{E}(\mathcal{P}_{3}) \cup \mathcal{E}(\mathcal{P}_{4}) $.

The discussion above shows that $\mathcal{E}(\mathcal{P}_{1})  \uplus \mathcal{E}(\mathcal{P}_{2})  \supseteq \mathcal{E}(\mathcal{P}_{3})  \uplus \mathcal{E}(\mathcal{P}_{4}) $. The same goes for the proof of the fact that $\mathcal{A}(\mathcal{P}_{1})  \uplus \mathcal{A}(\mathcal{P}_{2})  \supseteq \mathcal{A}(\mathcal{P}_{3})  \uplus \mathcal{A}(\mathcal{P}_{4}) $
\end{proof}

\begin{claim}
$\mathcal{P}_{3} \in \mathcal{S}$.
\end{claim}
\begin{proof}
By Claim~\ref{claim-supseteq}, we have
\begin{equation}\label{ineq-e}
e_{\mathcal{E} \cup \mathcal{A}}(\mathcal{P}_{1})+e_{\mathcal{E} \cup \mathcal{A}}(\mathcal{P}_{2}) \geq e_{\mathcal{E} \cup \mathcal{A}}(\mathcal{P}_{3})+e_{\mathcal{E} \cup \mathcal{A}}(\mathcal{P}_{4}).
\end{equation}
Since $\mathcal{P}_{1}, \mathcal{P}_{2} \in \mathcal{S}$, by (\ref{sufficiency}) and (\ref{ineq-hb}), we have
\begin{equation}\label{ineq-e-inverse}
\begin{split}
e_{\mathcal{E} \cup \mathcal{A}}(\mathcal{P}_{3})+e_{\mathcal{E} \cup \mathcal{A}}(\mathcal{P}_{4})
&\geq \sum_{X \in \mathcal{P}_{3}} h(X)-b(\cup \mathcal{P}_{3}) +\sum_{X \in \mathcal{P}_{4}} h(X)-b(\cup \mathcal{P}_{4}) ~~\text{(by (\ref{sufficiency}))}  \\
&\geq  \sum_{X \in \mathcal{P}_{1}} h(X)-b(\cup \mathcal{P}_{1}) +\sum_{X \in \mathcal{P}_{2}} h(X)-b(\cup \mathcal{P}_{2}) ~~\text{(by (\ref{ineq-hb}))} \\
&=e_{\mathcal{E} \cup \mathcal{A}}(\mathcal{P}_{1})+e_{\mathcal{E} \cup \mathcal{A}}(\mathcal{P}_{2}) ~~\text{(since $\mathcal{P}_{1}, \mathcal{P}_{2} \in \mathcal{S}$).}
\end{split}
\end{equation}
Combining (\ref{ineq-e}) and (\ref{ineq-e-inverse}), we have all $\geq$ in (\ref{ineq-e}) and (\ref{ineq-e-inverse}) should be $=$.  Hence, we have   $\mathcal{E}(\mathcal{P}_{1})  \uplus \mathcal{E}(\mathcal{P}_{2})  = \mathcal{E}(\mathcal{P}_{3})  \uplus \mathcal{E}(\mathcal{P}_{4}) $ and $e_{\mathcal{E} \cup A}(\mathcal{P}_{3})=\sum_{X \in \mathcal{P}_{3}} h(X)-b(\cup \mathcal{P}_{3})$. Since $\mathcal{P}_{1}, \mathcal{P}_{2} \in \mathcal{S}$, $Y$ enters both $\cup \mathcal{P}_{1}$ and $\cup \mathcal{P}_{2}$ and thus $ Y \in \mathcal{E}(\mathcal{P}_{1}) \cap \mathcal{E}(\mathcal{P}_{2})$. Since $\mathcal{E}(\mathcal{P}_{1})  \uplus \mathcal{E}(\mathcal{P}_{2})  = \mathcal{E}(\mathcal{P}_{3})  \uplus \mathcal{E}(\mathcal{P}_{4}) $, we have $Y \in \mathcal{E}(\mathcal{P}_{3})$ and thus $Y$ enters $\cup \mathcal{P}_{3}$ or $Y \subseteq \cup \mathcal{P}_{3}$. Note that $\cup \mathcal{P}_{3}= (\cup \mathcal{P}_{1} )\cap (\cup \mathcal{P}_{2} )$. Since $Y$ enters both $\cup \mathcal{P}_{1}$ and $\cup \mathcal{P}_{2}$, we have $Y \nsubseteq \cup \mathcal{P}_{3}$ and thus $Y$ enters  $\cup \mathcal{P}_{3}$. Hence, $\mathcal{P}_{3} \in \mathcal{S}$.
\end{proof}

It follows from the discussion above that for any $\mathcal{P}_{1}, \mathcal{P}_{2} \in \mathcal{S}$, there always exists a $\mathcal{P}_{3} \in S$ such that $\cup \mathcal{P}_{3}= (\cup \mathcal{P}_{1}) \cap (\cup \mathcal{P}_{2})$. Since $|\mathcal{S}|$ is finite, we can find a $\mathcal{P}_{0} \in \mathcal{S} $ such that $\cup \mathcal{P}_{0}= \cap \{ \cup \mathcal{P}\}_{\mathcal{P} \in \mathcal{S}}$. Since $Y$ enters $\cup \mathcal{P}_{0}$, we have $Y \cap (\cup \mathcal{P}_{0}) \neq \emptyset$ and suppose $y \in Y \cap (\cup \mathcal{P}_{0})$. Then for any $\mathcal{P} \in \mathcal{S}$, $y \in \cup \mathcal{P}$. Orient the hyperedge  $Y$ to the dyperedge $(Y-y,y)$. Then for any $\mathcal{P} \in \mathcal{S}$, $(Y-y, y)$  enters $\cup \mathcal{P}$ and $e_{(\mathcal{E}-Y)\cup (\mathcal{A} + (Y-y,y))}(\mathcal{P})=e_{\mathcal{E} \cup \mathcal{A}}(\mathcal{P})$. Hence, (\ref{sufficiency}) still holds after the orientation.

\section{Acknowledgement}

The author was supported by the National Natural Science Foundation of China (No. 12201623) and the Natural Science Foundation of Jiangsu Province, China (No. BK20221105).

\end{document}